\documentclass[fleqn,preprint]{elsarticle}
\makeatletter
\def\ps@pprintTitle{%
 \let\@oddhead\@empty
 \let\@evenhead\@empty
 \def\@oddfoot{\centerline{\thepage}}%
 \let\@evenfoot\@oddfoot}
\makeatother

\usepackage{latexsym,amsmath,amssymb,amsthm}
\usepackage{graphicx}
\usepackage{epsfig,epsf,psfrag,epstopdf}
\usepackage{color,xcolor}
\usepackage{comment}
\usepackage{algorithm}
\usepackage{algpseudocode}
\usepackage{url}
\usepackage{enumerate}

\usepackage{array,booktabs}
\usepackage{siunitx}

\usepackage{eqparbox}
\newdimen{\algindent}
\algnewcommand{\LeftComment}[1]{\Statex \hspace{\algindent}
  \(\triangleright\) #1}

\newtheorem{lemma}{Lemma}

\theoremstyle{definition}
\newtheorem{remark}{Remark}

\newcommand{\be}{\begin{equation}}
\newcommand{\ee}{\end{equation}}

\newcommand{\ba}{\begin{aligned}}
\newcommand{\ea}{\end{aligned}}

\newcommand{\bea}{\begin{eqnarray}}
\newcommand{\eea}{\end{eqnarray}}

\newcommand{\cplus}{\overline{\mathbb{C}^{+}}}

\newcommand{\tensor}[1]{\rm{\mathbf{#1}}}

\def\nist{Information Technology Laboratory,
  National Institute of Standards and Technology,
  Boulder, CO 80305, USA}

\def\njit{Department of Mathematics Sciences,
New Jersey Institute of Technology,\\
Newark, New Jersey 07102, USA}

\def\odu{Department of Mathematics and Statistics,
Old Dominion University,\\
Norfolk, VA 23529, USA}

\def\csrc{Beijing Computational Science Research Center,
Beijing 100084, China}

\def\papertitle{Evaluation of Abramowitz functions in the right half of the complex plane}

\def\submitdate{June 10, 2019}

\renewcommand{\today}{\submitdate}	

\begin{document}

\begin{frontmatter}

\title{\papertitle}



\author[nist]{Zydrunas Gimbutas}
\address[nist]{\nist}
\ead{zydrunas.gimbutas@nist.gov}
\author[njit]{Shidong Jiang}
\address[njit]{\njit}
\ead{shidong.jiang@njit.edu}
\author[odu,csrc]{Li-Shi Luo}
\address[odu]{\odu}
\address[csrc]{\csrc}
\ead{lluo@odu.edu}

\begin{abstract}
  A numerical scheme is developed for the evaluation of Abramowitz
  functions $J_n$ in the right half of the complex plane.  For
  $n=-1,\, \ldots,\, 2$, the scheme utilizes series expansions for
  $|z|<1$ and asymptotic expansions for $|z|>R$ with $R$ determined by
  the required precision,
and modified Laurent series expansions which
  are precomputed via a least squares procedure to approximate $J_n$
  accurately and efficiently on each sub-region in the intermediate
  region $1\le |z| \le R$. For $n>2$, $J_n$ is evaluated via a
  recurrence relation. The scheme achieves nearly machine precision
  for $n=-1, \ldots, 2$, with the cost about four times of evaluating
  a complex exponential per function evaluation.
\end{abstract}
  

\end{frontmatter}

\section{Introduction} 
\label{introduction}

The Abramowitz functions $J_n$ of order $n$, defined by
%
\be
J_n(z)=\int_0^\infty t^n e^{-t^2-z/t} dt , \quad n\in \mathbb{Z},
\label{intrep}
\ee
are frequently encountered in kinetic theory (cf., \textit{e.g.},
\cite{Cercignani2000,Gross1959pof}), where the integral equations
resulting from linearization of the Boltzmann equation have these
functions (cf., \textit{e.g.}, \cite{Cercignani2000, Gross1959pof,
  LiW2015cnf, jiang2016jcp}) as the kernels.  The $n$-th order
Abramowitz function $J_n$ satisfies the third order ODE
\cite{abramowitz1953jmp, handbook}
\be
zJ_n^{'''}-(n-1)J_n^{''}+2J_n=0
\label{jnode}
\ee
and the recurrence relations
\be\label{recurrence0}
J_n'(z)=-J_{n-1}(z),
\ee
\be\label{recurrence1}
2J_n(z)=(n-1)J_{n-2}(z)+zJ_{n-3}(z).
\ee

Research on Abramowitz functions is rather
limited. In~\cite{handbook}, about two pages of Section~27.5 are
devoted to Abramowitz functions, which contain series and asymptotic
expansions, originally developed in
\cite{abramowitz1953jmp,laporte1937pr,zahn1937pr}. In
\cite{cole1979jcp}, numerical computation of Abramowitz functions is
discussed when $z$ is a positive real number, and, in particular, it
is shown that the recurrence relation for $J_n$ is stable in both
directions. In \cite{macleod1992anm}, a more efficient and reliable
numerical algorithm using Chebyshev expansions has been developed for
the evaluation of $J_n$ ($n=0,1,2$) when $z$ is a positive real
number.

For time-dependent or time-harmonic problems in kinetic
theory, evaluation of Abramowitz functions with complex arguments is
often required. However, we are not aware of any work on the
evaluation of Abramowitz functions in complex domains.

In this paper, we develop an efficient and accurate numerical scheme
for the evaluation of Abramowitz functions when its argument $z$ is in
the right half of the complex plane (denoted as $\cplus =\{z\in
\mathbb{C} |\operatorname{Re}(z)\ge 0\}$) for $n\ge -1$. We first note
that Chebyshev expansions are not good representations in the complex
domain since Chebyshev polynomials are orthogonal polynomials only
when the argument is real.  Second, when $|z|$ is small, say, less
than $r$ for some $r>0$, a series expansion can be used to evaluate
$J_n(z)$ accurately with small number of terms.  Third, when $|z|$ is
large, say, greater than $R$ for some $R>0$, the truncated asymptotic
expansion can be used to evaluate $J_n(z)$ accurately.

We now consider the intermediate region $D=\{z\in\cplus\, |\, r\le |z|
\le R\}$, where neither the series expansion nor the asymptotic
expansion can be used to achieve the required precision. Since $0$ and
$\infty$ are the only singular points of the ODE \eqref{jnode}
satisfied by $J_n$, standard ODE theory \cite{coddington1955ode} shows
that $J_n$ is analytic in $\mathbb{C}\setminus \{0\}$. Thus, $J_n$
admits an infinite Laurent series representation in $D$ by theory of
complex variables \cite{ahlfors1978}. One may naturally ask whether
$J_n(z)$ can be well approximated by a truncated Laurent series in
$D$. It turns out that such approximation requires excessively large
number of terms to achieve high accuracy, even if we do not consider
the difficulty encountered when solving it numerically. Furthermore,
this global approximation is extremely ill-conditioned due to the fact
that $J_n$ behaves like an exponential function asymptotically, making
its dynamic range too large to be resolved numerically with high
accuracy and rendering the scheme useless.

We propose two techniques to deal with the extreme ill-conditioning associated
with the global approximation of $J_n$ in $D$. 
First, we extract out the leading factor in the asymptotic
expansion of $J_n(z)$ and make a change of variable as follows: \be
J_n(z)=\sqrt{\frac{\pi}{3}}\left(\frac{\nu}{3}\right)^{n/2} e^{-\nu}
U_n(\nu), \qquad \nu=3\left(\frac{z}{2}\right)^{2/3}.
\label{asymtransform}
\ee
It has been shown in \cite{abramowitz1953jmp,laporte1937pr} that
$U_n(\nu)$ also satisfies a third order ordinary differential equation
(ODE) with $0$ a regular singular point and $\infty$ an irregular
singular point. Thus, $U_n(\nu)$ is
analytic for $z\in D$ and therefore can be represented
by an infinite Laurent series in $\nu$ in the transformed domain.
The main advantage of working with
$U_n(\nu)$ instead of $J_n(z)$ is that $U_n(\nu)$ has much smaller dynamic
range and thus admits more accurate and efficient approximation.

Next, we divide the intermediate region $D$ into several sub-regions
$D_i=\{z\in \cplus\, |\, r_i\le |z| \le r_{i+1}\}$ ($i=0,\ldots, M-1$,
$r_0=r$, $r_M=R$). By symmetry, we may further restrict ourself to
consider the quarter-annulus domain $Q_i=\{z\in \mathbb{C}\, |\,
\operatorname{Re}(z)\ge 0, \operatorname{Im}(z)\ge 0, r_i\le |z| \le
r_{i+1}\}$ ($i=0,\ldots, M-1$, $r_0=r$, $r_M=R$).  On each sub-region
$Q_i$, we approximate $U_n(\nu)$ via a modified Laurent series in
$\nu$ where the coefficients of the series are obtained by solving a
least squares problem where the linear system is set up by matching
the function values with the values of the modified Laurent series
representation on a set of $N$ points on the boundary of $Q_i$. The
least squares problem is still ill-conditioned and the conditioning
becomes worse as $N$ increases, but its solution can be used to
produce very accurate approximation to the function being
approximated.

Here, we would like to remark that recently least squares method has been applied
to construct accurate and stable approximation for many classes of functions.
In \cite{barnett2008jcp}, it is used together with method of fundamental
solutions to solve boundary value problems for the Helmholtz equation.
In \cite{gonnet2011etna}, it is used to construct rational approximation
for functions on the unit circle.
In \cite{adcock2016frames,adcock2018frames}, it is shown that a wide class of functions
can be approximated in an accurate and well-conditioned manner using frames and the
least squares method. In \cite{greengard2018sisc}, the least squares method is used
to construct efficient and accurate sum-of-Gaussians approximations for a class of
kernels in mathematical physics. Needless to say, the least squares problem itself
has to be solved using suitable algorithms. Many such algorithms exist (see, for example,
\cite{demmel1997,golub2013,gu1999simaa,paige2006simaa,trefethen1997}).

For $n\ge 3$, we apply the recurrence relation \eqref{recurrence1} to
compute $J_n(z)$. We note that the recurrence relation only needs the
values of $J_n$ for $n=0, 1, 2$. Since many applications in kinetic
theory require the evaluation of $J_{-1}$, we provide the direct
evaluation of $J_{-1}$ as well via our scheme since it is more
efficient than using the recurrence relation.

Clearly, the scheme presented in this paper
may be applied to the accurate evaluation of a very broad class of
special functions in complex domains. Very often these special functions
satisfy an ODE with a finite number of singular points. Therefore, they
are analytic in complex domains excluding singular points and branch cuts. Complex
analysis then ensures that Laurent series is a suitable representation
to such functions in the domain. With a careful choice of the domain and
suitable transformation, the least squares method becomes a reliable
tool for constructing efficient, accurate and stable approximation
for these functions.

The paper is organized as follows. Section~\ref{sec:prelim} collects analytical
results used in the construction of the algorithm.
Section~\ref{sec:alg} discusses numerical algorithms for the evaluation
of Abramowitz functions. Section~\ref{sec:examples} illustrates
the performance and accuracy of the algorithm. The paper is concluded
with a short discussion on possible extensions and applications of the
work.

\section{Analytical apparatus}
\label{sec:prelim}

The series expansion of $J_n$ takes the form
\be
\label{seriesexpansion}
2J_n(z)=\sum_{k=0}^\infty(a_k^{(n)}\ln z+b_k^{(n)})z^k. 
\ee
For $n=1$, the coefficients can be found in
\cite[\S27.5.4]{handbook}
with $a_0^{(1)}=a_1^{(1)}=0$, $a_2^{(1)}=-1$, $b_0^{(1)}=1$, $b_1^{(1)}=-\sqrt{\pi}$,
$b_2^{(1)}=3(1-\gamma)/2$, and 
\be
\label{coefab} 
a_k^{(1)}= - \frac{2a_{k-2}^{(1)}}{k(k-1)(k-2)}, \qquad b_k^{(1)}= -
\frac{2b_{k-2}^{(1)}+(3k^2-6k+2)a_k^{(1)}}{k(k-1)(k-2)}, \quad k\ge 3, 
\ee 
where $\gamma\approx 0.577215664901532860606512$ is Euler's constant.
For $n=-1, 0$, the coefficients can be obtained from term-by-term
differentiation of \eqref{seriesexpansion}, together with \eqref{recurrence0}:
\be
\label{coefab0} 
a_k^{(n)}= -(k+1) a_{k+1}^{(n+1)},
\qquad b_k^{(n)}= -(k+1) b_{k+1}^{(n+1)} - a_{k+1}^{(n+1)}, \quad k\ge 0.
\ee
For $n=2$, the coefficients can be obtained from term-by-term integration 
of \eqref{seriesexpansion} together with $J_2(0)=\sqrt{\pi}/4$, i.e.,
$a_0^{(2)} = 0$, $b_0^{(2)}=\sqrt{\pi}/2$, and
\be
\label{coefab2} 
a_k^{(2)}= - \frac{a_{k-1}^{(1)}}{k},
\qquad b_k^{(2)}= - \frac{b_{k-1}^{(1)}}{k} + \frac{a_{k-1}^{(1)}}{k^2}, \quad k\ge 1.
\ee
We have the following lemma regarding the convergence of
the power series $\sum_{k=0}^\infty a_k^{(n)}z^k$
and $\sum_{k=0}^\infty b_k^{(n)}z^k$ in the series expansion \eqref{seriesexpansion}.

\begin{lemma}\label{seriesconv}
For $n=-1,\ldots,2$, the power series
$\sum_{k=0}^\infty a_k^{(n)}z^k$ and $\sum_{k=0}^\infty b_k^{(n)}z^k$
in \eqref{seriesexpansion} converge in $\mathbb{C}$.
\end{lemma}
\begin{proof}
  For $n=1$, direct calculation shows that
  \be\label{ak}
  a_{2k-1}=0, \quad a_{2k}^{(1)}=\frac{(-1)^k2}{(2k)!(k-1)!}, \quad k>0.
  \ee
  Thus, the radius of convergence for $\sum_{k=0}^\infty a_k^{(n)}z^k$
  is $\infty$ by the ratio test and the series converges for all complex
  numbers.
  We now split $\sum_{k=0}^\infty b_k^{(n)}z^k$ into the odd part and the even
  part:
  \be
  \sum_{k=0}^\infty b_k^{(1)}z^k = z\sum_{k=0}^\infty b_{2k+1}^{(1)}(z^2)^k
  +\sum_{k=1}^\infty b_{2k}^{(1)}(z^2)^k.
  \ee
  For the odd part, direct calculation shows
  \be\label{b2kp1}
  b_{2k+1}^{(1)}=\frac{(-2)^kb_1^{(1)}}{(2k+1)!(2k-1)!!}.
  \ee
  Using the root test and Stirling's formula for factorials \cite[p.~201]{ahlfors1978},
  we observe that the odd part converges for all complex numbers.
  For the even part, we claim that
  \be\label{b2kbound}
  |b_{2k}^{(1)}|<\frac{2}{[(k-1)!]^3}, \quad k\ge 1.
  \ee
  We prove \eqref{b2kbound} by induction. First, \eqref{b2kbound} holds for $k=1$
  by direct calculation. Now, assume \eqref{b2kbound} holds for $2k-2$, i.e.,
  \be
  |b_{2k-2}^{(1)}|<\frac{2}{[(k-2)!]^3}.
  \ee
  By \eqref{ak},
  it is easy to see that
  \be\label{akbound}
  |a_{2k}^{(1)}|<\frac{1}{2^k[(k-1)!]^3}, \quad k>1.
  \ee
  using the second equation in \eqref{coefab}, we have
  \be
  \ba
  |b_{2k}^{(1)}|&\le \frac{2|b_{2k-2}^{(1)}|}{2k(2k-1)(2k-2)}
  +\frac{3|a_{2k}^{(1)}|}{k-1}+\frac{2|a_{2k}^{(1)}|}{2k(2k-1)(2k-2)}\\
  &<\frac{2|b_{2k-2}^{(1)}|}{2k(2k-1)(2k-2)}+\frac{1}{[(k-1)!]^3}\\
  &<\frac{4}{2k(2k-1)(2k-2)[(k-2)!]^3}+\frac{1}{[(k-1)!]^3}\\
  &<\frac{2}{[(k-1)!]^3},
  \ea
  \ee
  where the first inequality follows from the triangle inequality, the second
  one follows from \eqref{akbound}, the third one follows from the induction
  assumption. Thus, the even part also converges for all complex numbers by
  the comparison and root
  tests, and Stirling's formula. Finally, the convergence of the power series
  for $n=-1, 0, 2$ follows from \eqref{coefab0}, \eqref{coefab2}, \eqref{ak},
  \eqref{b2kp1}, \eqref{b2kbound}, the comparison and root tests, and Stirling's formula.
\end{proof}

Even though \eqref{seriesexpansion} was originally derived under the
assumption that $z$ is positive real, it makes sense for any $z\neq
0$. Furthermore, it provides a natural analytic
continuation~\cite[p.~283]{ahlfors1978} of $J_n$ to $\mathbb{C}$ with
the branch cut along negative real axis and the principal branch for
$\ln z$ chosen to be, say, $\operatorname{Im}(\ln z)\in (-\pi,\pi]$.

The asymptotic expansion of $J_n$ is given by
\cite[\S27.5.8]{handbook}:
\be
\label{asymexpansion}
J_n(z) \sim \sqrt{\frac{\pi}{3}}\left(\frac{\nu}{3}\right)^{n/2}e^{-\nu}
\left(c_0^{(n)}+\frac{c_1^{(n)}}{\nu}+\frac{c_2^{(n)}}{\nu^2}+\cdots\right),
\quad z\rightarrow \infty ,
\ee
where $\nu=3(z/2)^{2/3}$, $c_0^{(n)}=1$, $c_1^{(n)}=(3n^2+3n-1)/12$, and
\be
\label{coefc}
\ba
12(k+2)c_{k+2}^{(n)} = & -(12k^2+36k-3n^2-3n+25)c_{k+1}^{(n)}\\
&+\frac{1}{2}(n-2k)(2k+3-n)(2k+3+2n)c_k^{(n)}, \quad k\ge 0.
\ea
\ee
Once again, \eqref{asymexpansion} was originally derived
under the assumption that $z$ is real
and positive \cite{abramowitz1953jmp,laporte1937pr}. One may, however,
verify that the expansion inside the parentheses on the right hand
side of \eqref{asymexpansion} is a formal solution to the third order
ODE satisfied by $U_n$ in \eqref{asymtransform}. Furthermore, the
exponential factor decays when
$\arg{z}\in(-\frac{3\pi}{4},\frac{3\pi}{4})$.  Hence,
\eqref{asymexpansion} is valid for any $z\in \cplus$ as $z\rightarrow
\infty$.

The following lemma is the theoretical foundation of our algorithm.
\begin{lemma}\label{maxprinciple}
  Suppose that $D \subset \mathbb{C}$ is a closed bounded domain that
  does not contain the origin and the function $f$ is analytic in $D$.
  Let $L(z)=\sum_{k=-N_1}^{N_2}c_k z^k$. Then
  \begin{enumerate}[(i)]
  \item if $|f(z)-L(z)|\le \epsilon$ for $z\in \partial D$, then
    $|f(z)-L(z)|\le \epsilon$ for $z\in D$;
  \item if $|f(z)-L(z)|/|f(z)|\le \epsilon$ for $z\in\partial D$
    and $f$ has no zeros in $D$, then
    $|f(z)-L(z)|/|f(z)|\le \epsilon$ for $z\in D$.
  \end{enumerate}
\end{lemma}
\begin{proof}
  This follows from the analyticity of $L(z)$ on $D$ and the maximum
  principle \cite[p.~133]{ahlfors1978}.
\end{proof}
\section{Numerical Algorithms}\label{sec:alg}
\subsection{Series and asymptotic expansions}
As we have shown in Lemma~\ref{seriesconv}, the coefficients $a_k^{(n)}$ and $b_k^{(n)}$
in~\eqref{coefab}--\eqref{coefab2} decay very rapidly and the
corresponding series expansions converge for any $z \ne 0$. However,
they cannot be used for numerical calculation for large $|z|$
due to cancellation errors and increasing number of terms for achieving the
desired precision. Thus, we will use the series expansions
only for $|z|<1$ (i.e., $r=1$). In this region, both power series
$\sum_{k=0}^\infty a_k^{(n)} z^k$ and $\sum_{k=0}^\infty b_k^{(n)} z^k$ converge exponentially fast
and very few terms are needed to reach the desired precision.

The coefficients $c_k^{(n)}$ in~\eqref{coefc} diverge rapidly and the
asymptotic expansion \eqref{asymexpansion} has to be truncated in order
to be of any use. For any truncated asymptotic expansion, it is well-known
that its accuracy increases as $|z|$ increases. For a prescribed precision $\epsilon_{\rm mach}$,
one needs to determine $N_a$ --- the number of terms in the truncated
series, and $R$ with $|z|>R$ the applicable region of the truncated
series. This is straightforward to determine numerically.
We have found that $N_a=18$ and $R=120$ are sufficient to achieve
IEEE extended precision for $J_n$ ($n=-1,\ldots,2$).

\subsection{Construction of the modified Laurent series for the intermediate region}
\label{sec:precomp}

We now discuss the evaluation of $J_n$ in the intermediate
region $D=\{z\in\cplus\, |\, r\le |z| \le R\}$.
First, it is easy to see that $J_n(\bar{z})=\bar{J_n}(z)$ from its
integral representation \eqref{intrep}. Thus,
we will only discuss the evaluation of $J_n$ in the first quadrant
$Q=\{z\in\mathbb{C}\, |\, r\le |z| \le R, , 0\le \arg{z}\le \frac{\pi}{2}\}$.
As discussed in the introduction, it is very difficult to directly
approximate $J_n(z)$ in $Q$ due to its large dynamic range. 
We use the transformation~\eqref{asymtransform} and consider the approximation
of $U_n(\nu)$ instead, $U_n$ has a very small dynamic range.
Figure~\ref{fig1} shows ${\rm Log}_{10}$ of $|J_{-1}(z)|$ in $Q$
on the left and $|U_{-1}(z)$ in $Q$ on the right, where the left panel
shows that the magnitude of $J_{-1}(z)$
ranges from $10^{-20}$ to $10^{0}$, and the right panel shows that the magnitude
of $U_{-1}(z)$ ranges from $0.95$ to $1$. Other $J_n(z)$ and $U_n(z)$ ($n=0, 1, 2$)
exhibit similar pattern.
Thus, we will consider the evaluation
of $U_n(\nu)$ in $Q$.

\begin{figure}[t]
\centering
  \includegraphics[height=40mm]{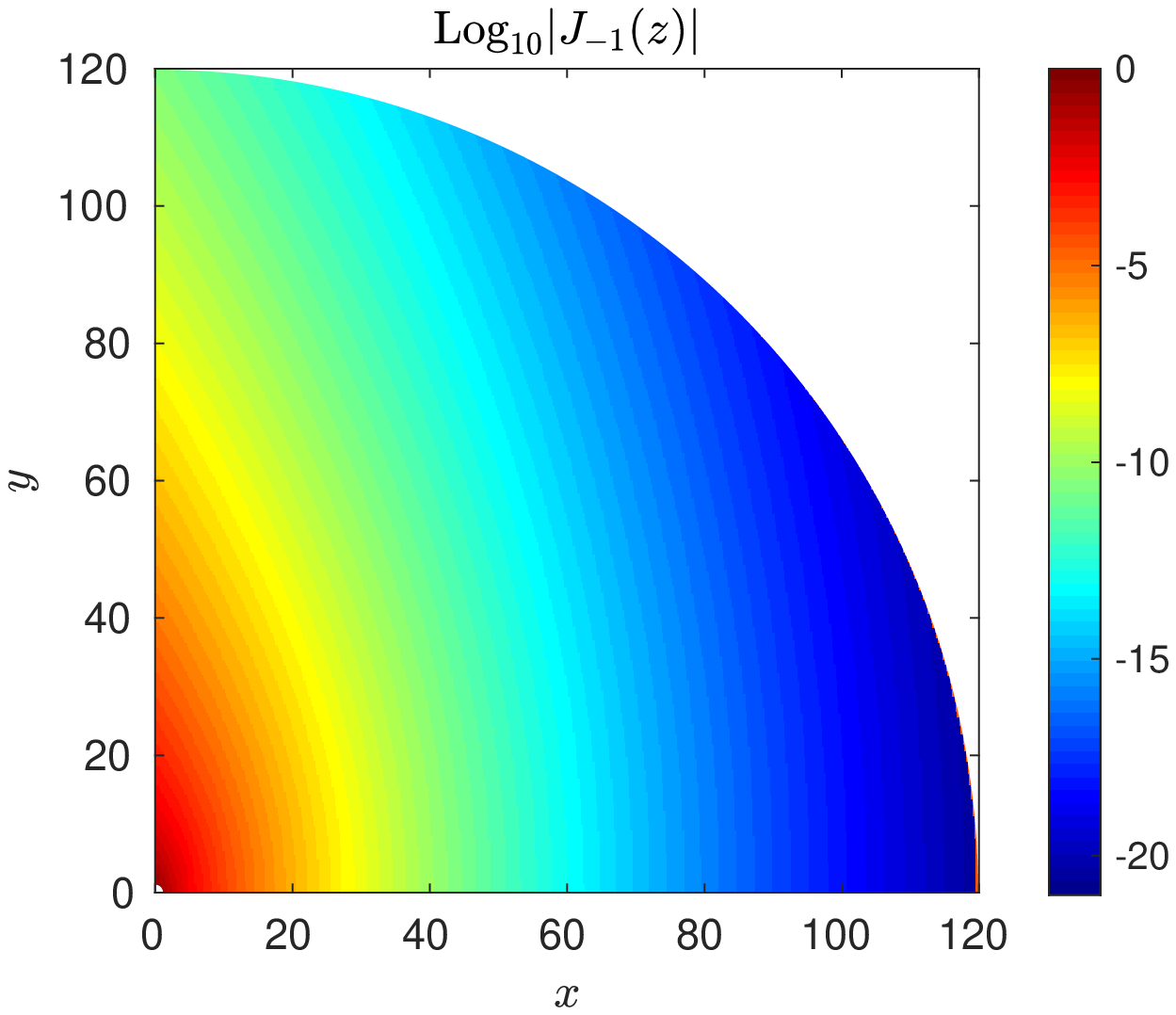}
  \includegraphics[height=40mm]{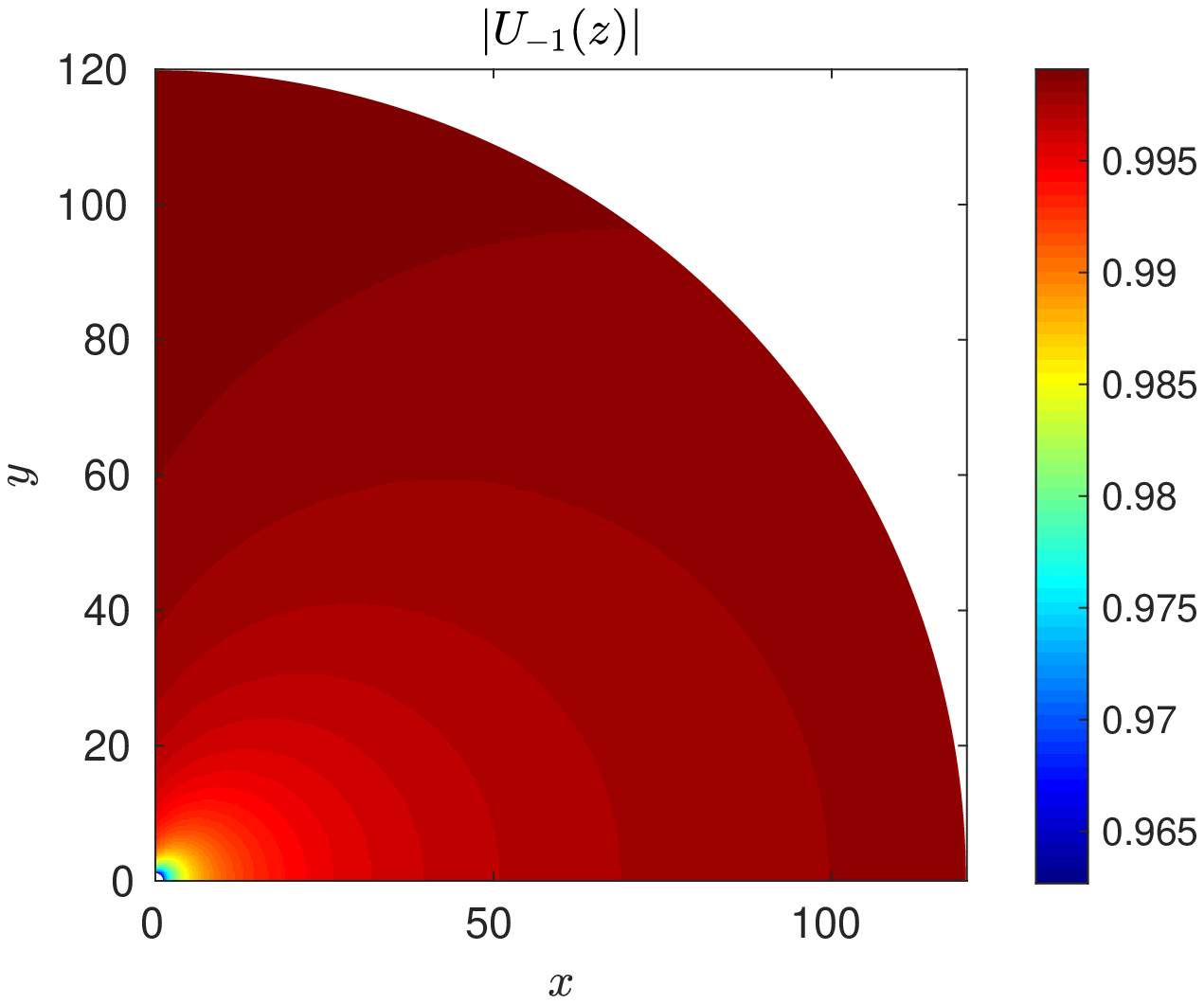}
  \caption{Dynamic ranges of $J_{-1}(z)$ and $U_{-1}(z)$ in $Q$. For comparison purpose,
  both figures are plotted in the variable $z$.}
\label{fig1}
\end{figure}

To this end, we divide $Q$ into several quarter-annulus domains:
\be
Q_i=\{z\in \mathbb{C} | r_i\le |z| \le r_{i+1}, 0\le \arg{z}\le \frac{\pi}{2}\},
\,\, i=0,\ldots, M-1, \,\, r_0=1, r_M=R.
\ee
We will try to approximate $U_n(\nu)$ in each $Q_i$ via a modified
Laurent series
\be
U_n(\nu)\simeq L_n^{(i)}(\nu) = \sum_{k=-N_1}^{N_2}d_k^{(i)} \nu^k, \qquad z\in Q_i.
\label{laurentapprox}
\ee
As noted before, $U_n(\nu)$ satisfies a third order ODE with $0$ and
$\infty$ as the only singular points
\cite{abramowitz1953jmp,laporte1937pr}.  Thus, $U_n(\nu)$ is analytic
in $Q_i$. By Lemma~\ref{maxprinciple}, in order to guarantee the
accuracy of the approximation in the whole domain $Q_i$, it is
sufficient to ensure the same accuracy is achieved on the boundary of
$Q_i$, \textit{i.e.},
\be\label{errorbound1}
\left|U_n(\nu)-\sum_{k=-N_1}^{N_2}d_k^{(i)} \nu^k\right|\le \epsilon,
\qquad z \in \partial Q_i.
\ee

The error-bound in \eqref{errorbound1} is achieved by solving the least squares problem
\be
\tensor{A} \tensor{d}^{(i)}=\tensor{f},
\label{leastsquares}
\ee
where 
\be
A_{jk}=\nu_j^k, \qquad f_j = U_n(\nu_j), \quad j=1,\ldots,4N_b ,
\ee
where $\nu_j=3(z_j/2)^{2/3}$, and $z_j$ are chosen to be the images of
Gauss-Legendre nodes on each segment of $\partial Q_i$, $N_b$ is
chosen to ensure that the error of approximation of $U_n(\nu)$ by the
corresponding Legendre polynomial interpolation on each segment of
$\partial Q_i$ is bounded by $\epsilon$.
The right hand side $\tensor{f}$ in \eqref{leastsquares} is computed
\textit{via} symbolic software system {\sc Mathematica} to at least
$50$ digits. In other words, we do not use the actual analytic Laurent
series to approximate $U_n$ on each quarter-annulus $Q_i$. Instead, a
numerical procedure is applied to find much more efficient
``modified'' Laurent series for approximating $U_n$ on each $Q_i$.

The linear system~\eqref{leastsquares} is ill-conditioned.  However,
since we always use $\tensor{d}^{(i)}$ in the modified Laurent series to
evaluate $U_n$, we will obtain high accuracy in function evaluation in
the entire sub-region as long as the residual of the least squares
problem~\eqref{leastsquares} is small by the maximum principle.

The least squares solver also reveals the numerical
rank of $\tensor{A}$, which is used to obtain the optimal value of
$N_{\rm T}=N_2-N_1+1$, the total number of terms in the modified
Laurent series.  It is then straightforward to use a simple search to
find the value for $N_1$, which completes the algorithm
for finding a nearly optimal and highly accurate modified Laurent
series approximation for $U_n$ in $Q_i$.

\begin{remark}
  We would like to emphasize that the modified Laurent series may not
  be unique, but this non-uniqueness has no effect on the accuracy of
  the approximation.
\end{remark}

\begin{remark}
  We have computed the integrals
  \be\label{indef}
  I_n=\int_{\partial Q} \frac{J_n'(z)}{J_n(z)}dz
  =-\int_{\partial Q} \frac{J_{n-1}(z)}{J_n(z)}dz
  \ee
  for $n=-1,\ldots,2$ and found numerically that they are all close to zero.
  By the argument principle \cite[p.~152]{ahlfors1978},
  we have
  \be
  I_n=2\pi i(Z_n-P_n),
  \ee
  where $Z_n$ and $P_n$ denote respectively the number of zeros and poles of $J_n(z)$
  inside $\partial Q$.
  Since $J_n(z)$ is analytic in $Q$, it has no poles in $Q$, i.e., $P_n=0$.
  Thus, the fact that $I_n$ is very close to zero shows that $Z_n=0$, that is,
  $J_n$ has no zeros in $Q$.
  Further numerical investigation shows that $|U_n(\nu)|$ ranges from $0.95$ to $1.7$
  on $\partial Q$.
  Combining these two facts, we conclude that the absolute error bound on the
  approximation of $U_n$ gives
  roughly the same relative error bound.
\end{remark}

\subsection{Evaluation of $J_n$ for $n=-1,\, \ldots,\, 2$}

Once the coefficients of modified Laurent series for each sub-region
are obtained and stored, the evaluation of $J_n(z)$ is
straightforward. That is, we first compute $|z|$ to decide on which
region the point lies, then use the proper representation to evaluate
$J_n(z)$ accordingly.  We summarize the algorithm for calculating
$J_{n}(z)$ for $z\in \cplus$, $n=-1,\, \ldots,\, 2$ in
Algorithm~\ref{alg1}.

\begin{algorithm} 
\caption{Evaluation of $J_n(z)$ for $z\in\cplus$}
\label{alg1}
\begin{algorithmic}
  \Procedure{Abram}{$z$,$f$}
\LeftComment{Input parameter: $z$ - the complex number for which the Abramowitz function $J_n$
is to be evaluated.}
\LeftComment{Output parameter: $f$ - the value of  Abramowitz function $J_n(z)$.}
\State \textbf{assert} $\mbox{Re}(z)\ge 0$.
\If {$|z|\le 1$} \Comment $z$ is in the series expansion region.
    \State Use the series expansion \eqref{seriesexpansion} to evaluate $f=J_n(z)$.
\ElsIf {$|z| \ge 120$} \Comment $z$ is in the asymptotic region.
    \State Set $\nu=3(z/2)^{2/3}$.
    \State Use the asymptotic expansion \eqref{asymexpansion} to compute $U_n(\nu)$.
    \State Set $f=\sqrt{\frac{\pi}{3}}\, \left(\frac{\nu}{3}\right)^{n/2}\, e^{-\nu}\, U_n(\nu)$.
\Else \Comment $z$ is in the intermediate region.
    \State Set $\nu=3(z/2)^{2/3}$.
    \State Use a precomputed modified Laurent series expansion \eqref{laurentapprox} to compute $U_n(\nu)$.
    \State Set $f=\sqrt{\frac{\pi}{3}}\, \left(\frac{\nu}{3}\right)^{n/2}\, e^{-\nu}\, U_n(\nu)$.
    \EndIf
    \EndProcedure
\end{algorithmic}
\end{algorithm}
\begin{remark}
  All these expansions can be converted into a polynomial of certain transformed variable
  multiplying with some factor. We use Horner's method \cite[\S4.6.4]{knuth1997} to evaluate
  the polynomial in optimal arithmetic operations.
\end{remark}

\begin{remark}
The accuracy of $J_n(z)$ deteriorates
as $|z|$ increases since the condition number of evaluating the
exponential function $e^{-\nu}$ is $|\nu|$. This is unavoidable in any
numerical scheme as the phenomenon is related to physical
ill-conditioning of evaluating $J_n(z)$ for the argument with large
magnitude.
\end{remark}

\subsection{Evaluation of $J_n$ for $n>2$}

In \cite{cole1979jcp}, it is shown that \eqref{recurrence1} is stable
in both directions when $z$ is a positive real number. We have
implemented the forward recurrence to evaluate $J_n(z)$ for $n>2$. We
have not observed any numerical instability during our numerical tests
for $z\in\cplus$.

\section{Numerical results}
\label{sec:examples}

We have performed numerical experiments on a laptop with a 2.10GHz
Intel Core i7-4600U processor and 4GB of RAM.

For the series expansion
\eqref{seriesexpansion}, a straightforward calculation shows that
$18$ terms in $\sum b_k^{(n)} z^k$ and $9$ nonzero
terms in $\sum a_k^{(n)} z^k$ are needed to reach the IEEE extended
precision ($\epsilon_{\rm mach}\approx 1.09 \cdot 10^{-19}$)
for $J_n$ ($n=-1,\ldots,2$).
For the asymptotic expansion
\eqref{asymexpansion}, we find that it is sufficient to choose
$N_a=18$, $R=120$ for the IEEE extended precision.
All coefficients are precomputed with $50$ digit
precision.

For the intermediate region, we divide $|z|$ on $[1, 120]$ into three
subintervals $[1, 3]$, $[3, 15]$, $[15, 120]$ and $Q$ into $Q_1$,
$Q_2$, $Q_3$, respectively. We use quadruple precision to
carry out the precomputation step and solve the least squares problem 
with $10^{-20}$ accuracy. We have found that for
$Q_1$ we need $N_2=11$, $N_T=30$ for
$J_0$ and $J_1$, $N_2=10$, $N_T=32$ for $J_{-1}$, and $N_2=11$,
$N_T=32$ for $J_2$.  For all four functions $J_n$
($n=-1,0,1,2$), we need $N_2=0$, $N_T=30$ for $Q_2$ and
$N_2=0$, $N_T=20$ for $Q_3$. The coefficients of modified Laurent
series for $J_n$ ($n=-1,0,1,2$) on $Q_i$ ($i=1,2,3$) are listed
in Tables~\ref{tabm1a}--\ref{tab2c} in~\ref{coeftables}.

\begin{remark}
  The coefficients in Tables~\ref{tabm1a}--\ref{tab2c}
  for $Q_2$ and $Q_3$ do not have small norms. However, for $Q_2$,
  $\left|\frac{1}{\nu}\right|\le \frac{1}{3(3/2)^{(2/3)}}=0.254\ldots$;
  and for $Q_3$,
  $\left|\frac{1}{\nu}\right|\le \frac{1}{3(15/2)^{(2/3)}}\approx 0.087$.
  It is easy to see that terms $c_j\left(\frac{1}{\nu}\right)^j$
  decrease as $j$ increases. Alternatively, we could consider
  the Laurent series of the form $\sum \tilde{c}_j \left(\frac{\nu_i}{\nu}\right)^j$
  with $\nu_i=3(r_i/2)^{(2/3)}$ ($r_i$ is the lower bound for $|z|$ in $Q_i$).
  Then the coefficient vector $\tilde{c}$ will have small norm, as required in
  \cite{barnett2008jcp,adcock2016frames}. However,
  this corresponds to the column scaling in the least squares matrix and almost
  all methods for solving the least squares problems do column normalization.
  Thus, it has no effect on the accuracy of the solution and stability
  of the algorithm.
\end{remark}

\begin{remark}
The partition of the sub-regions is by no means optimal or
unique. There is an obvious trade-off between the number of
sub-regions and the number of terms in the modified Laurent series
(the total number of terms in the modified Laurent expansion slightly
increases as the regions gets closer to the origin). This suggests
that one may use a finer partition for the regions closer to the
origin. We have tried to divide the intermediate region into $14$
regions with $Q_i=\{z\in\cplus|(\sqrt{2})^{i-1}\le |z|\le (\sqrt{2})^i\}$
($i=1,\ldots,14$), and we observe that only $20$ terms
are needed for all regions. However, our numerical
experiments indicate that the partition has little effect on the
overall performance (i.e., speed and accuracy) of the algorithm.
\end{remark}

We first check the accuracy of Algorithm~\ref{alg1}. The reference function values
are calculated via {\sc Mathematica} to at least 50 digit accuracy.
The error is measured in terms of maximum relative error, i.e.,
\[E=\max_i\frac{|\hat{J}_n(z_i)-\tilde{J}_n(z_i)|}{|\tilde{J}_n(z_i)},\]
where $\tilde{J}_n(z_i)=e^{\nu_i} J_n(z_i)$ ($\nu_i=3(z_i/2)^{2/3}$)
is the reference value of the scaled Abramowitz function
computed via {\sc Mathematica},
and $\hat{J}_n(z_i)$ is the value computed via our algorithm.
The points $z_i$ are sampled randomly with uniform distribution
in both its magnitude and angle in $\cplus$. Table~\ref{table1} lists
the errors for evaluating $\tilde{J}_n$ ($n=-1,0,1,2$) in
various regions, where we observe that the errors are within $10\epsilon_{\rm mach}$
with the machine epsilon $\epsilon_{\rm mach}\approx 2.22\times 10^{-16}$ for IEEE double
precision. In general, the errors in the first intermediate
region $Q_1$ are slightly bigger due to mild cancellation errors.

\begin{table}[!ht]
  \caption{The relative $L^\infty$ error of Algorithm~\ref{alg1} over
    $100,000$ uniformly distributed
    random points in $\cplus$. The reference value is computed via {\sc Mathematica} to
    at least 50 digit accuracy. $S$ denotes the series expansion region and $A$ denotes
    the asymptotic expansion region.
  }
\sisetup{
  table-format = 1.1e-2,
  table-auto-round=true,
  tight-spacing=true
}
\centering
\begin{tabular}{lSSSSS}
  \toprule
  &     {S}     &    ${Q_1}$ &     ${Q_2}$ &    ${Q_3}$ &    {A}\\
  \midrule
${J_{-1}}$  &  1.52E-15&  2.10E-15&  4.38E-16&  6.38E-16&  8.64E-16\\
\midrule
${J_{ 0}}$  &  1.31E-15&  2.43E-15&  2.21E-16&  2.18E-16&  2.17E-16\\
\midrule
${J_{ 1}}$  &  1.06E-15&  2.40E-15&  4.65E-16&  5.96E-16&  7.97E-16\\
\midrule
${J_{ 2}}$  &  1.20E-15&  2.93E-15&  5.55E-16&  8.44E-16&  1.17E-15\\
 \bottomrule
 \end{tabular}
\label{table1}
\end{table}

Since all three representations (i.e., modified Laurent series, series and
asymptotic expansions) mainly involve polynomials of degree less than 30,
the algorithm takes
about constant time per function evaluation in $\cplus$. 
We have tested the CPU time of Algorithm~\ref{alg1} for evaluating $\tilde{J_n}(z)$
and compared it with
that of evaluating the complex exponential $e^z$.
The results are shown in Table~\ref{table2}.
We observe that on average the CPU time of each
function evaluation is about four times of that for the complex exponential
evaluation.

\begin{table}[!ht]
  \caption{The total CPU time in seconds for evaluating $\tilde{J_n(z)}$ using
    Algorithm~\ref{alg1} and the exponential function $e^z$
    over $100,000$ uniformly distributed
    random points in $\cplus$. 
  }
\sisetup{
  table-format = 1.1e-2,
  table-auto-round=true,
  tight-spacing=true
}
\centering
\begin{tabular}{lSSSSS}
  \toprule
  & ${J_{-1}(z)}$ & ${J_0(z)}$ & ${J_1(z)}$ & ${J_2(z)}$& ${e^z}$ \\
  \midrule
  ${T}$ &  3.60E-02&  3.20E-02&  3.60E-02&  3.20E-02& 8.00E-03\\
 \bottomrule
 \end{tabular}
\label{table2}
\end{table}

For $n>2$, we have tested the stability of
the forward recurrence relation \eqref{recurrence1}
for evaluating $J_n$ in $\cplus$. Numerical experiments indicate
that the forward recurrence relation is stable
in $\cplus$. The relative errors are shown in Table~\ref{table3}
for a typical run.
\begin{table}[!ht]
  \caption{The maximum relative error for evaluating $J_{100}$ using
    the forward recurrence relation \eqref{recurrence1} over $100,000$
    uniformly distributed random points in the domain $\{z\in \mathbb{C}| \operatorname{Re}(z)\ge0, 0<|z|<1000\}$. The reference
    value is calculated using {\sc Mathematica} with $240$-digit precision
    arithmetic.
  }
\sisetup{
  table-format = 1.1e-2,
  table-auto-round=true,
  tight-spacing=true
}
\centering
\begin{tabular}{SSSSS}
  \toprule
      {S} &  ${Q_1}$ &  ${Q_2}$ &  ${Q_3}$ & {A}\\
      \midrule
  1.29E-15&  2.86E-15&  1.26E-15&  2.03E-15&  3.66E-15\\
 \bottomrule
 \end{tabular}
\label{table3}
\end{table}

\section{Conclusions and further discussions}

We have designed an efficient and accurate algorithm for the
evaluation of Abramowitz functions $J_n$ in the right half of the
complex plane.  Some useful observations in the design of the
algorithm are applicable for evaluating many other special functions
in the complex domain.  First, it is better to pull out the leading
asymptotic factor from the given function when $|z|$ is large. Second,
the maximum principle reduces the dimensionality of the approximation
problem by one. Third, the least squares scheme is generally a
reliable and accurate method to find an approximation of a prescribed
form. That is, analytical representations should be used with caution
even if they are available, as they often lead to large cancellation
error or very inefficient approximations or both.

Finally, though we have used truncated Laurent series representation for approximating
Abramowitz functions in the intermediate region, there are many other
representations for function approximations. This includes truncated series expansion,
rational functions (see, for example, \cite{gonnet2011etna}), etc. We have actually tested
the truncated series expansion in the sub-region (i.e., $Q_1$) closest to the origin
for $J_n$. Our numerical experiments indicate that the performance is about the same
as the one presented in this paper.

\section*{Acknowledgments}

S.~Jiang was supported by the National Science Foundation under grant
DMS-1720405, and by the Flatiron Institute, a division of the Simons
Foundation.  
L.-S.~Luo was supported by the National Science Foundation under grant
DMS-1720408. 
The authors would like to thank Vladimir Rokhlin at Yale
University for his unpublished pioneer work on the evaluation of Hankel
functions in the complex plane and Manas Rachh at the Flatiron
Institute, Simons Foundation for helpful discussions. Certain
commercial software products and equipment are identified in this
paper to foster understanding. Such identification does not imply
recommendation or endorsement by the National Institute of Standards
and Technology, nor does it imply that the software
products and equipment identified are necessarily the best available for the
purpose.

\appendix 

\section{Zeros of $J_n(z)$}\label{app1}

We have used {\sc NIntegrate} in {\sc Mathematica} to evaluate $I_n$
defined in \eqref{indef}. When {\sc WorkingPrecision} is set to $100$,
$|I_n|$ are about $10^{-59}$ for $n=-1,0,1,2$. When it is set to $200$,
the values of $|I_n|$ decrease to $10^{-160}$. By the argument priniciple,
$I_n$ can only take integral multiples of $2\pi i$. Thus, the numerical
calculation clearly shows that $J_n$ ($n=-1,0,1,2$) have no zeros
in the intermediate region $Q$. Analytically, we can only show that
$J_n$ has no zeros in the sector $|\operatorname{arg}(z)|\le \frac{\pi}{4}$.
The proof is presented below.
\begin{lemma}\label{lem1}
  If $z_0\in \mathbb{C}$ is a zero of $J_n(z)$, then also $\bar{z}_0$.
\end{lemma}
\begin{proof}
  By the integral representation of $J_n(z)$ in \eqref{intrep}, we have
  $J_n(\bar{z}_0)=\bar{J}_n(z_0)$ and the lemma follows.
\end{proof}
\begin{lemma}\label{lem3}
  Suppose that $n\ge 0$. Then
  $J_n(z)$ has no zero in the sector $|\arg{z}|\le \frac{\pi}{4}$.
\end{lemma}
\begin{proof}
  Let $z_0=x_0+iy_0 \in \cplus$ be a zero of $J_n(z)$. Then by Lemma~\ref{lem1},
  $\bar{z}_0$ is also a zero of $J_n(z)$. Consider functions $f(t)=J_n(z_0t)$ and
  $g(t)=J_n(\bar{z}_0t)$. Then $f(1)=g(1)=0$, and $f$, $g$ and their derivatives
  decay exponentially fast to $0$ as $t\rightarrow \infty$ by the asymptotic
  expansion \eqref{asymexpansion}.

  The differential equation \eqref{jnode} implies that
  \be\label{fode}
  tf'''(t)-(n-1)f''(t)+2z_0^2f(t)=0,
  \ee
  \be\label{gode}
  tg'''(t)-(n-1)g''(t)+2\bar{z}_0^2g(t)=0.
  \ee

  Multiplying both sides of \eqref{fode} by $g$, integrating both sides
  from $1$ to $\infty$, and performing integration by parts, we obtain
  \be\label{int1}
  \ba
  0&=\int_1^\infty [tf'''g-(n-1)f''g+2z_0^2fg]dt\\
  &=tgf''\bigg\vert_1^\infty-\int_1^{\infty} f''(g+g't)dt+
  \int_1^\infty [-(n-1)f''g+2z_0^2fg]dt\\
  &=\int_1^\infty [-tf''g'-nf''g+2z_0^2fg]dt\\
  &=\int_1^\infty [-tf''g'+nf'g'+2z_0^2fg]dt.
  \ea
  \ee
  Similarly,
  \be\label{int2}
  0  =\int_1^\infty [-tf'g''+nf'g'+2\bar{z}_0^2fg]dt.
  \ee
  Moreover,
  \be\label{int3}
  \ba
  \int_1^\infty [-tf'g''-tf'g'']dt&=-\int_1^\infty td(f'g')\\
  &=-tf'g'\bigg\vert_1^\infty+\int_1^{\infty} f'g'dt\\
  &=f'(1)g'(1)+\int_1^{\infty} f'g'dt.
  \ea
  \ee
  Adding \eqref{int1}, \eqref{int2} and using \eqref{int3} to simplify the result, we obtain
  \be\label{int4}
  0=f'(1)g'(1)+(2n+1)\int_1^{\infty} f'g'dt+2(z_0^2+\bar{z}_0^2)\int_1^{\infty} fgdt.
  \ee
  Rearranging \eqref{int4}, we have
  \be\label{int5}
  4(y_0^2-x_0^2)\int_1^\infty |J_n(z_0t)|^2dt=
  |z_0|^2\left|J'_n(z_0)\right|^2+
  (2n+1)|z_0|^2\int_1^\infty \left|J'_n(z_0t)\right|^2dt.
  \ee
  Since the right side of \eqref{int5} and the integral
  on its left side are both positive, we must have $y_0^2-x_0^2>0$
  and the lemma follows.
\end{proof}

\begin{lemma}
  $J_n(z)$ has no zero in $D=\{z\in \cplus\vert |z|>R\}$, where $R$
  is sufficiently large.
\end{lemma}
\begin{proof}
  Subtracting \eqref{int2} from \eqref{int1}, we have
  \be\label{lem4.1}
  0=\int_1^\infty t(f'g''-f''g')dt+2(z_0^2-\bar{z}_0^2)\int_1^\infty fgdt.
  \ee
  That is,
  \be\label{lem4.2}
  4x_0y_0\int_1^\infty \left|J_n(z_0t)\right|^2dt
  =|z_0|^2\int_1^\infty \operatorname{Im}
  \left(\bar{z}_0tJ_{n-1}(z_0t)J_{n-2}(\bar{z}_0t)\right) dt.
  \ee
  In the domain $D$, $J_n(z)$ is well approximated by the leading
  term of its asymptotic expansion. Let $z_0=r_0 e^{i\theta_0}$ with $r_0>0$
  and $\theta_0\in [-\pi/2,\pi/2]$. Substituting the leading terms
  of the asymptotic expansions into both sides of \eqref{lem4.2}
  and simplifying the resulting expressions, we obtain
  \be
  \sin(2\theta_0)\sim-\sin(2\theta_0/3).
  \ee
  In other words, two sides of \eqref{lem4.2} have opposite sign unless they
  are both equal to zero, \textit{i.e.}, unless $\theta_0=0$ or $z_0$ is a positive
  real number. However, $J_n(x)>0$ when $x>0$, as seen from its integral
  representation \eqref{intrep}. And the lemma follows.
\end{proof}

\section{The coefficients of modified Laurent series for $J_n$}
\label{coeftables}

We list the coefficients $c_j$ of modified Laurent series for evaluating
$J_n$ ($n=-1$, 0, 1, and 2) on each quarter-annulus domain $Q_i$ ($i=1$,
2, and 3) in Tables~\ref{tabm1a}--\ref{tab2c}. That is,
\be
  J_n(z)
  \approx 
  \sqrt{\frac{\pi}{3}}
  \left(\frac{\nu}{3}\right)^{n/2} 
  e^{-\nu}
  \nu^{N_2} 
  \sum_{j=0}^{N_T-1} c_j 
  \left(\frac{1}{\nu}\right)^j 
  ,
\label{b1}
\ee
where $\nu := 3 \left(\frac{z}{2}\right)^{2/3}$. \eqref{b1} is
obtained by combining \eqref{asymtransform} and \eqref{laurentapprox},
and rewriting the modified Laurent series as a power series in
$\frac{1}{\nu}$ by pulling out the factor $\nu^{N_2}$.

\begin{table}[htbp!]
   \caption{The coefficients $c_j$ ($j=0,\, \ldots,\, 31$) of the
     modified Laurent series given by \eqref{b1} 
     to evaluate $J_{-1}(z)$ to $19$-digit precision in $Q_1 :=
     \{z\in \mathbb{C}\, |\, \operatorname{Re}(z) \ge 0,
     \operatorname{Im}(z)\ge 0, 1 \le |z| \le 3 \}$. $N_2=10$.
} 
\scriptsize \centering
\begin{tabular}{SS}
  \toprule
      {real part} & {imaginary part}\\
      \midrule
      0.50840463208260678152D-17 & -0.17460815299463749948D-15 \\
     -0.74591223502642620660D-14 &  0.12462600200296453012D-13 \\
      0.51034244856324824207D-12 & -0.29429847146968217669D-12 \\
     -0.15527853485027100709D-10 &  0.26315851430676356796D-12 \\
      0.26441404512287963095D-09 &  0.13983475139768244907D-09 \\
     -0.24748763871353093363D-08 & -0.37421319823017115933D-08 \\
      0.48226858274090904108D-08 &  0.54340128932660141072D-07 \\
      0.21625355372586607508D-06 & -0.50872099870851161398D-06 \\
     -0.36871705117848123797D-05 &  0.30134016655759593920D-05 \\
      0.34628404889507030160D-04 & -0.73910823070405617219D-05 \\
      0.99977737459069660694D+00 & -0.57603083624530025151D-04 \\
     -0.82327858162819693045D-01 &  0.84976858037261402153D-03 \\
      0.61974789354573766566D-03 & -0.60513938190315115982D-02 \\
      0.56615182294768079637D-01 &  0.30290788934912727755D-01 \\
     -0.13513677999109029679D+00 & -0.11595801511190682178D+00 \\
      0.20971815296188580167D+00 &  0.34917394460827128828D+00 \\
     -0.13559302399958735143D+00 & -0.83031652814884108813D+00 \\
     -0.39989898854107271642D+00 &  0.15322617865070516148D+01 \\
      0.17398271638333840850D+01 & -0.20720742240832378654D+01 \\
     -0.38218064277297175142D+01 &  0.16655645078067718066D+01 \\
      0.57404644363343330931D+01 &  0.36441373700438752895D+00 \\
     -0.60218557453822568030D+01 & -0.36699889432071554424D+01 \\
      0.38664098293378463250D+01 &  0.65632365306504714202D+01 \\
     -0.25697949139671290942D+00 & -0.71730334400727831101D+01 \\
     -0.26138368668366285752D+01 &  0.52101729871789289259D+01 \\
      0.33128062049048194583D+01 & -0.22744720239035355439D+01 \\
     -0.22947550138820496670D+01 &  0.21824868540130264477D+00 \\
      0.97998841946268481518D+00 &  0.43081914671714156557D+00 \\
     -0.23273278868918701241D+00 & -0.30881694364539401656D+00 \\
      0.13573816724649184659D-01 &  0.10103591470624091510D+00 \\
      0.64717665310787895482D-02 & -0.16204976273553372187D-01 \\
     -0.11353082240496407813D-02 &  0.91111645509511076869D-03 \\
 \bottomrule
 \end{tabular}
 \label{tabm1a}
 \end{table}

\begin{table}[htbp!]
    \caption{Similar to Table~\ref{tabm1a}, $c_j$ ($j=0,\, \ldots,\, 29$)
      for $J_{-1}(z)$ in $Q_2=\{z\in \mathbb{C}\, |\,
      \operatorname{Re}(z)\ge 0, \operatorname{Im}(z)\ge 0, 3\le |z|
      \le 15\}$. $N_2=0$.
%
} 
\scriptsize \centering
 \begin{tabular}{SS}
 \toprule
 {real part} & {imaginary part}\\
 \midrule
      0.99999999999996165301D+00 &  0.14180683234758492536D-12 \\
     -0.83333333315888343156D-01 & -0.18355475502542401539D-10 \\
      0.34722202099214306218D-02 &  0.66429512090231781628D-09 \\
      0.55459217936935525195D-01 &  0.19209070325965982796D-07 \\
     -0.17477009309488548835D+00 & -0.27235872415655243493D-05 \\
      0.47557985079285319878D+00 &  0.12329339800149018587D-03 \\
     -0.12044719601488244381D+01 & -0.33379989131254858384D-02 \\
      0.24160534977076998585D+01 &  0.59920404647268786033D-01 \\
      0.71401934124020221324D+00 & -0.69764699651584141417D+00 \\
     -0.60367540682210374145D+02 &  0.38607796239007672158D+01 \\
      0.60545135048209986187D+03 &  0.32429279475615845398D+02 \\
     -0.45946367108344566727D+04 & -0.10841949093356820460D+04 \\
      0.28358573752155457724D+05 &  0.14766714227455119633D+05 \\
     -0.13554840952273842275D+06 & -0.13629699320708838668D+06 \\
      0.42722335885416276983D+06 &  0.93640538562551055857D+06 \\
     -0.18717734419017137932D+06 & -0.49014636853552340206D+07 \\
     -0.76674698133130508647D+07 &  0.19293996919633486842D+08 \\
      0.56621620119877490002D+08 & -0.53379687761932413868D+08 \\
     -0.24544626054098569413D+09 &  0.76969984306318530811D+08 \\
      0.73652406083022339655D+09 &  0.11898870845017090857D+09 \\
     -0.15200293963699011585D+10 & -0.11167744707665950837D+10 \\
      0.18157636339201652460D+10 &  0.37001812450286450398D+10 \\
      0.23697005105214074056D+09 & -0.76973235212132828329D+10 \\
     -0.60541865274209691412D+10 &  0.10509437050190981895D+11 \\
      0.13447347591183417529D+11 & -0.82869660102657042317D+10 \\
     -0.16538600326905832899D+11 &  0.87889333133786055548D+09 \\
      0.12108038654949012813D+11 &  0.58320160548830925371D+10 \\
     -0.45881323770532016082D+10 & -0.64591210758282531847D+10 \\
      0.33559769561348792357D+09 &  0.30012974946895292083D+10 \\
      0.21590442067376607526D+09 & -0.51553627638896435829D+09 \\
 \bottomrule
 \end{tabular}
 \label{tabm1b}
 \end{table}

 \begin{table}[htbp!]
     \caption{Similar to Table~\ref{tabm1a}, $c_j$ ($j=0,\, \ldots,\,
       19$) for $J_{-1}(z)$ in $Q_3=\{z\in \mathbb{C}\, |\,
       \operatorname{Re}(z)\ge 0, \operatorname{Im}(z)\ge 0, 15\le |z|
       \le 120\}$. $N_2=0$.
} 
\scriptsize \centering
\begin{tabular}{SS}
 \toprule
 {real part} & {imaginary part}\\
 \midrule
      0.10000000000000000211D+01 &  0.17867305969317471010D-16 \\
     -0.83333333333337062447D-01 & -0.97723166437483860903D-14 \\
      0.34722222219662307873D-02 &  0.18575099531415563550D-11 \\
      0.55459105063779291569D-01 & -0.17036760654072501033D-09 \\
     -0.17476652435372606609D+00 &  0.72097356819624208386D-08 \\
      0.47552180369947961103D+00 &  0.45886722797104214745D-07 \\
     -0.12045748284986630605D+01 & -0.23432576523368445886D-04 \\
      0.24476460069464141708D+01 &  0.14488404302693181855D-02 \\
     -0.19443570247379529707D+00 & -0.51284106279947756551D-01 \\
     -0.44775070512394599808D+02 &  0.11973398083848165562D+01 \\
      0.42163459709409223079D+03 & -0.18732584217083190217D+02 \\
     -0.30990846226113832847D+04 &  0.18064501304516396811D+03 \\
      0.20913884916390585368D+05 & -0.59296211153624558148D+03 \\
     -0.12895996355054874785D+06 & -0.10821482660282026169D+05 \\
      0.67085295525681611041D+06 &  0.19808862750865237678D+06 \\
     -0.26337877182586616150D+07 & -0.16813332113288193078D+07 \\
      0.67096341894817561042D+07 &  0.85931950401381414532D+07 \\
     -0.76571281908120841443D+07 & -0.26572627858717182234D+08 \\
     -0.59803448026875748509D+07 &  0.44801684284187004703D+08 \\
      0.19209322347765871037D+08 & -0.30066013610259277074D+08 \\
 \bottomrule
 \end{tabular}
 \label{tabm1}
 \end{table}

 \begin{table}[htbp!]
   \caption{The coefficients $c_j$ ($j=0,\, \ldots,\, 29$) of the
     modified Laurent series  given by \eqref{b1} to evaluate $J_{0}(z)$ 
     to $19$-digit 
     precision in $Q_1 := \{z\in \mathbb{C}\, |\,
     \operatorname{Re}(z)\ge 0, \operatorname{Im}(z)\ge 0, 1\le |z|
     \le 3\}$. $N_2=11$.
} \scriptsize \centering
 \begin{tabular}{SS}
 \toprule
 {real part} & {imaginary part}\\
 \midrule
     -0.90832607641433626723D-16 & -0.12971716857438253177D-15 \\
      0.12389804620230878343D-14 &  0.12374086345769475560D-13 \\
      0.18925665936446973863D-12 & -0.43594042425001909808D-12 \\
     -0.93436124699082728782D-11 &  0.71610240171455947215D-11 \\
      0.21005471373426356192D-09 & -0.31570720100508497322D-10 \\
     -0.27921469604412283831D-08 & -0.10267754766776108185D-08 \\
      0.22210697286892781643D-07 &  0.25281961277933484578D-07 \\
     -0.71494613675879873372D-07 & -0.30771058781715872427D-06 \\
     -0.63954539217286436591D-06 &  0.24264961215641857661D-05 \\
      0.11477599934330755236D-04 & -0.12692783111400141826D-04 \\
     -0.94447601385518738118D-04 &  0.36582122360442313063D-04 \\
      0.10005227131416389419D+01 &  0.42527679632933790645D-04 \\
     -0.85416339695541973974D-01 & -0.11668274275597700974D-02 \\
      0.92680088758957499003D-01 &  0.75525130431164108091D-02 \\
     -0.12830962183399732914D+00 & -0.31999279672322587569D-01 \\
      0.18262801902460105636D+00 &  0.10112601789451786253D+00 \\
     -0.22086524963618477505D+00 & -0.24794786230219646677D+00 \\
      0.16700979504519707012D+00 &  0.47574461283438653640D+00 \\
      0.62230841342939847177D-01 & -0.70546898276037268906D+00 \\
     -0.46160590435660301333D+00 &  0.77405336431438665012D+00 \\
      0.86282905145667261861D+00 & -0.54883285003677036633D+00 \\
     -0.10157725409050061178D+01 &  0.89653554833115686588D-01 \\
      0.81127042933258073193D+00 &  0.34223426976661392785D+00 \\
     -0.40588870780744941328D+00 & -0.50414950527612358637D+00 \\
      0.70646295770216518095D-01 &  0.39020053902964327900D+00 \\
      0.62576278207207491689D-01 & -0.18665248819340615755D+00 \\
     -0.56069463746156990540D-01 &  0.51369345024189480453D-01 \\
      0.20854565059593331201D-01 & -0.49622460298807761080D-02 \\
     -0.37736411371630856848D-02 & -0.10706620427594972634D-02 \\
      0.24658062816300990462D-03 &  0.24793548656986345025D-03 \\
 \bottomrule
 \end{tabular}
 \label{tab0a}
 \end{table}

 \begin{table}[thp!]
    \caption{Similar to Table~\ref{tab0a}, $c_j$ ($j=0,\, \ldots,\,
      29$) for $J_{0}(z)$ in $Q_2 := \{z\in \mathbb{C}\, |\,
      \operatorname{Re}(z)\ge 0, \operatorname{Im}(z)\ge 0, 3\le |z|
      \le 15\}$. $N_2=0$.
%
} \scriptsize \centering
    \begin{tabular}{SS}
      \toprule
          {real part} & {imaginary part}\\
          \midrule
      0.99999999999988637217D+00 & -0.86635400939375232846D-13 \\
     -0.83333333323131499972D-01 &  0.22519321671024025250D-10 \\
      0.86805555689429826898D-01 & -0.20734497580441697938D-08 \\
     -0.11815206514175737947D+00 &  0.96339270441630372077D-07 \\
      0.17969333057187777654D+00 & -0.23057043651727157918D-05 \\
     -0.24337342169790375842D+00 &  0.98098354524518453957D-05 \\
      0.14764429473872620763D-01 &  0.12782321450918857049D-02 \\
      0.23309627937902507555D+01 & -0.51073707172260220779D-01 \\
     -0.16480981490288764106D+02 &  0.11291585827876199238D+01 \\
      0.92305117544447684520D+02 & -0.17262096251706180600D+02 \\
     -0.51600417922753718015D+03 &  0.19328752764190907274D+03 \\
      0.32464931691139105961D+04 & -0.15901053336317843819D+04 \\
     -0.22314060512920598345D+05 &  0.90599663123843335441D+04 \\
      0.14604708231321899614D+06 & -0.26233323039837437629D+05 \\
     -0.81145513740976803311D+06 & -0.98104807756320344084D+05 \\
      0.35457971304583219348D+07 &  0.18294948704040264718D+07 \\
     -0.11084057351725538629D+08 & -0.13170160809055557556D+08 \\
      0.17946742587727355394D+08 &  0.62986605485232028006D+08 \\
      0.35800094666248698896D+08 & -0.21614451081432945159D+09 \\
     -0.37304955427569800721D+09 &  0.52194699072674118229D+09 \\
      0.14479420488451058753D+10 & -0.75867692917492941354D+09 \\
     -0.35951788816564166863D+10 &  0.22618398696580184017D+08 \\
      0.60001115021506662033D+10 &  0.31006219604826621451D+10 \\
     -0.60780548332669974319D+10 & -0.87891548726761143921D+10 \\
      0.15576242961336566425D+10 &  0.13885136549901956758D+11 \\
      0.55250083598538675266D+10 & -0.13622496533575828228D+11 \\
     -0.92625760934489411655D+10 &  0.75732780185372212772D+10 \\
      0.69543035636142439747D+10 & -0.12850366647773799533D+10 \\
     -0.25630570764916110006D+10 & -0.85541751450691424086D+09 \\
      0.33863352297553708594D+09 &  0.36939690751181780209D+09 \\
 \bottomrule
 \end{tabular}
 \label{tab0b}
 \end{table}

 \begin{table}[htbp!]
    \caption{Similar to Table~\ref{tab0a}, $c_j$ ($j=0,\, \ldots,\,
      19$) for $J_{0}(z)$ in $Q_3 := \{z\in \mathbb{C}\, |\,
      \operatorname{Re}(z)\ge 0, \operatorname{Im}(z)\ge 0, 15\le |z|
      \le 120\}$. $N_2=0$.
} \scriptsize \centering
 \begin{tabular}{SS}
 \toprule
 {real part} & {imaginary part}\\
 \midrule
      0.99999999999999996930D+00 &  0.17593864033911935746D-16 \\
     -0.83333333333319900902D-01 & -0.16846147898292977950D-15 \\
      0.86805555553423816181D-01 & -0.11453394704160148063D-11 \\
     -0.11815200602347672298D+00 &  0.23048385969290284678D-09 \\
      0.17968950772370040285D+00 & -0.21945385998958748345D-07 \\
     -0.24323777650814324772D+00 &  0.12087954626873876383D-05 \\
      0.11700570395152938411D-01 & -0.38103854798729283635D-04 \\
      0.23745206848674586241D+01 &  0.41346400173493025458D-03 \\
     -0.16758837066520389957D+02 &  0.20437630577696960942D-01 \\
      0.88966545765421650942D+02 & -0.12064317238786081373D+01 \\
     -0.39494374065352553320D+03 &  0.33912080437719049806D+02 \\
      0.14059897264393353350D+04 & -0.62474820464709927286D+03 \\
     -0.40451355322086019935D+04 &  0.80876358650678996996D+04 \\
      0.19384543104359811933D+05 & -0.74559599790322239386D+05 \\
     -0.20517688352680136030D+06 &  0.48086289900895412320D+06 \\
      0.17043521840830888384D+07 & -0.20518647360918521409D+07 \\
     -0.87705077033153779372D+07 &  0.49903982905075629408D+07 \\
      0.26797195290998643893D+08 & -0.30874807345514436382D+07 \\
     -0.43448795451180985546D+08 & -0.14198190720585577590D+08 \\
      0.26693074419888988636D+08 &  0.25674511583029811722D+08 \\
 \bottomrule
 \end{tabular}
 \label{tab0c}
 \end{table}

 \begin{table}[thbp!]
   \caption{The coefficients $c_j$ ($j=0,\, \ldots,\, 29$) of the
     modified Laurent series given by \eqref{b1} for evaluation of
     $J_{1}(z)$ to $19$-digit precision in $Q_1 := \{ z\in
     \mathbb{C}\, |\, \operatorname{Re}(z)\ge 0,
     \operatorname{Im}(z)\ge 0, 1\le |z| \le 3 \}$. $N_2=11$.
} \scriptsize \centering
 \begin{tabular}{SS}
 \toprule
 {real part} & {imaginary part}\\
 \midrule
      0.11005198342846485755D-15 & -0.69497479897694901798D-16 \\
     -0.10253717390952836750D-13 &  0.57344733061298400754D-15 \\
      0.35541980746147250213D-12 &  0.17138035947739436987D-12 \\
     -0.57065663426773316925D-11 & -0.80097752680384861353D-11 \\
      0.21377942402322032801D-10 &  0.17745761143488866634D-09 \\
      0.92504177773563681659D-09 & -0.23491690282446912066D-08 \\
     -0.21958132206773784869D-07 &  0.18715255541162079236D-07 \\
      0.26744146813435088020D-06 & -0.60538192727561566571D-07 \\
     -0.21417964338884026393D-05 & -0.55166777471060148467D-06 \\
      0.11605229803403698059D-04 &  0.10089206827157776627D-04 \\
     -0.36902286245955926331D-04 & -0.85734523630464858844D-04 \\
      0.99998886026520537047D+00 &  0.49984760395796500059D-03 \\
      0.41764834336748872867D+00 & -0.21759152835287886250D-02 \\
     -0.12880435876254801279D+00 &  0.72384623466304053406D-02 \\
      0.10001797075393494104D+00 & -0.18204681378409207336D-01 \\
     -0.12326994533416519599D+00 &  0.32494324768891981532D-01 \\
      0.19390815724013910177D+00 & -0.31048056008057556412D-01 \\
     -0.30446220948476603076D+00 & -0.28094371375086694729D-01 \\
      0.40217178304394830919D+00 &  0.18701432829068936487D+00 \\
     -0.39045411378013014641D+00 & -0.42965195739247365992D+00 \\
      0.20299486067051564492D+00 &  0.64270959413495860386D+00 \\
      0.10086985423793876517D+00 & -0.67712707054748086683D+00 \\
     -0.34706311627097173040D+00 &  0.48943085630911667832D+00 \\
      0.39717127772830854281D+00 & -0.20388651987798637536D+00 \\
     -0.27627970190658614495D+00 & -0.36920639378264876881D-02 \\
      0.12093911370868755832D+00 &  0.67168398413971524804D-01 \\
     -0.28845724988884421284D-01 & -0.45442309954198940539D-01 \\
      0.97253793671434469328D-03 &  0.15236745276873557399D-01 \\
      0.11989911484170176670D-02 & -0.25400926250086296020D-02 \\
     -0.20436565348663071365D-03 &  0.14672057879876671250D-03 \\
 \bottomrule
 \end{tabular}
 \label{tab1a}
 \end{table}

\begin{table}[htbp!]
  \caption{Similar to Table~\ref{tab1a}, $c_j$ ($j=0,\, \ldots,\,
    29$) for $J_{1}(z)$ in $Q_2 := \{z\in \mathbb{C}\, |\,
    \operatorname{Re}(z)\ge 0, \operatorname{Im}(z)\ge 0, 3\le |z|
    \le 15 \}$. $N_2=0$.
%
} \scriptsize \centering
 \begin{tabular}{SS}
 \toprule
 {real part} & {imaginary part}\\
 \midrule
      0.10000000000001559822D+01 & -0.64432580975613771082D-13 \\
      0.41666666663765045792D+00 & -0.30929290380316585308D-11 \\
     -0.12152777575602031881D+00 &  0.13847887495450779512D-08 \\
      0.64139599010730684601D-01 & -0.11804143890505963878D-06 \\
      0.19340333876868250914D-01 &  0.52525129103659222882D-05 \\
     -0.31085396288117458325D+00 & -0.14209942615180352958D-03 \\
      0.14076112393497740595D+01 &  0.22843270818239797021D-02 \\
     -0.53034603582858591137D+01 & -0.12508426902104267053D-01 \\
      0.16843969837042581097D+02 & -0.39633859151173122745D+00 \\
     -0.30896103962008743537D+02 &  0.13366867242921989470D+02 \\
     -0.12354078658896338072D+03 & -0.23100296557839310161D+03 \\
      0.16486028729189802772D+04 &  0.27798011316789725915D+04 \\
     -0.79365600281962572617D+04 & -0.25085379438345641690D+05 \\
     -0.17876338162032969792D+04 &  0.17337711302201102756D+06 \\
      0.36042195374432098049D+06 & -0.90932303898462350515D+06 \\
     -0.33818329308700649502D+07 &  0.34283955443872702745D+07 \\
      0.19449202897749494834D+08 & -0.74625499598352743942D+07 \\
     -0.79047083372251398893D+08 & -0.61704036445636743642D+07 \\
      0.22787889419550420217D+09 &  0.13542454969531036697D+09 \\
     -0.41550638595000604177D+09 & -0.65483373280300662671D+09 \\
      0.17290712369617688112D+09 &  0.19664895233379210138D+10 \\
      0.16763777971399520871D+10 & -0.40005075227615193294D+10 \\
     -0.63097117018939634689D+10 &  0.51442671230430695393D+10 \\
      0.12639623708614918540D+11 & -0.24250533919233200087D+10 \\
     -0.16021424838185937866D+11 & -0.51063223260105245852D+10 \\
      0.12194470039177973074D+11 &  0.12801295658115048467D+11 \\
     -0.36617569740928099053D+10 & -0.13906995324951658335D+11 \\
     -0.20882430849265640704D+10 &  0.82351413910371916318D+10 \\
      0.22256324759689985206D+10 & -0.23606999259817906485D+10 \\
     -0.57357275466466452587D+09 &  0.18100167963268264995D+09 \\
 \bottomrule
 \end{tabular}
 \label{tab1b}
 \end{table}

 \begin{table}[htbp!]
   \caption{Similar to Table~\ref{tab1a}, $c_j$ ($j=0,\, \ldots,\,
       19$) for $J_{1}(z)$ in $Q_3 := \{z\in \mathbb{C}\, |\,
       \operatorname{Re}(z) \ge 0, \operatorname{Im}(z)\ge 0, 15\le |z|
       \le 120 \}$. $N_2=0$.
} \scriptsize \centering
 \begin{tabular}{SS}
 \toprule
 {real part} & {imaginary part}\\
 \midrule
      0.10000000000000000088D+01 & -0.37104682094436741073D-16 \\
      0.41666666666665693268D+00 &  0.10633105786560679943D-13 \\
     -0.12152777777532530135D+00 & -0.81783483309751920964D-12 \\
      0.64139660206844395055D-01 & -0.53206746309599215652D-10 \\
      0.19340376146506349534D-01 &  0.14714309085259108266D-07 \\
     -0.31092901473600760433D+00 & -0.12847633757844741159D-05 \\
      0.14108230204421526148D+01 &  0.64029804309267135469D-04 \\
     -0.53814287035192935174D+01 & -0.19729163937458920069D-02 \\
      0.18099040506545928510D+02 &  0.33959266046710551528D-01 \\
     -0.44222521101771005259D+02 & -0.46443129149364017364D-01 \\
     -0.49281712505609892221D+02 & -0.14673338698274539790D+02 \\
      0.19120731039799297866D+04 &  0.44654655766491527724D+03 \\
     -0.18480296407139556113D+05 & -0.76529257176182418914D+04 \\
      0.11949451283301751133D+06 &  0.88037548115504996527D+05 \\
     -0.51411364057002663824D+06 & -0.70563400281206830770D+06 \\
      0.10973535392819828712D+07 &  0.39099452669156576703D+07 \\
      0.17331435565199135031D+07 & -0.14355007661133573735D+08 \\
     -0.19127195748032646177D+08 &  0.31757253245227946371D+08 \\
      0.50801574443329963657D+08 & -0.33700657946301331333D+08 \\
     -0.47910288059234253994D+08 &  0.61171907037609011958D+07 \\
 \bottomrule
 \end{tabular}
 \label{tab1c}
 \end{table}

 \begin{table}[htbp!]
    \caption{The coefficients $c_j$ ($j=0,\, \ldots,\, 31$) of the
      modified Laurent series given by \eqref{b1} to evaluate
      $J_{2}(z)$ to $19$-digit precision in $Q_1=\{z\in
      \mathbb{C}\, |\, \operatorname{Re}(z)\ge 0,
      \operatorname{Im}(z)\ge 0, 1\le |z| \le 3\}$. $N_2=11$.
%
} \scriptsize \centering
\begin{tabular}{SS}
\toprule
 {real part} & {imaginary part}\\
 \midrule
      0.31866632685819612221D-16 &  0.62278738969830137987D-16 \\
      0.23374202488114714431D-15 & -0.58368186829092031391D-14 \\
     -0.12464763262921601144D-12 &  0.20259199528722770999D-12 \\
      0.54938914867389395195D-11 & -0.30688169770355401691D-11 \\
     -0.12158575842106281373D-09 &  0.20706119567919592298D-12 \\
      0.16018003273928560165D-08 &  0.88216842473436341238D-09 \\
     -0.11941475881449767865D-07 & -0.18815539800292924303D-07 \\
      0.15306868790345339446D-07 &  0.22567034551573649710D-06 \\
      0.80407254481241384544D-06 & -0.17772860498030423720D-05 \\
     -0.11465040286971344849D-04 &  0.88993097363922792729D-05 \\
      0.92698417450695281624D-04 & -0.17153044258356809874D-04 \\
      0.99948171991549382451D+00 & -0.15046194458696453245D-03 \\
      0.14187062252412931802D+01 &  0.18349141691729387994D-02 \\
     -0.12647835873372535821D+00 & -0.11428276705995908659D-01 \\
      0.18603598111474663325D+00 &  0.50646899433597524505D-01 \\
     -0.28764948748580855321D+00 & -0.17268243476553171463D+00 \\
      0.37671933372380252436D+00 &  0.46496076480920682474D+00 \\
     -0.28673562730306685162D+00 & -0.99215999495574618374D+00 \\
     -0.25260802503698437799D+00 &  0.16511988403580568251D+01 \\
      0.14171261472971538418D+01 & -0.20381244501795426809D+01 \\
     -0.29457967249708858073D+01 &  0.15782077203635430607D+01 \\
      0.40284406518941590462D+01 & -0.37723354335564292819D-01 \\
     -0.38128950013920692643D+01 & -0.19907064876530093167D+01 \\
      0.22475249180098600179D+01 &  0.33343987769848875142D+01 \\
     -0.29845507995469994980D+00 & -0.32351706541913757406D+01 \\
     -0.89544287444481992984D+00 &  0.20454467870089365942D+01 \\
      0.10175639567579240708D+01 & -0.77474351850574639868D+00 \\
     -0.58788778865800093021D+00 &  0.83503119281933322786D-01 \\
      0.20009319020081182763D+00 &  0.77689730095477593636D-01 \\
     -0.35965432090839096074D-01 & -0.43797603367743883119D-01 \\
      0.16769420558534530117D-02 &  0.95756365430182522746D-02 \\
      0.27256710553195448121D-03 & -0.76612682266254092889D-03 \\
 \bottomrule
 \end{tabular}
 \label{tab2a}
 \end{table}

 \begin{table}[htbp!]
     \caption{Similar to Table~\ref{tab2a}, $c_j$ ($j=0,\, \ldots,\,
       29$) for $J_{2}(z)$ in $Q_2=\{z\in \mathbb{C}\, |\,
       \operatorname{Re}(z)\ge 0, \operatorname{Im}(z)\ge 0, 3\le |z|
       \le 15\}$. $N_2=0$.
%
} \scriptsize \centering
 \begin{tabular}{SS}
 \toprule
 {real part} & {imaginary part}\\
 \midrule
      0.99999999999998061808D+00 &  0.16248252113309315678D-12 \\
      0.14166666666829275403D+01 & -0.23049278428420443693D-10 \\
     -0.12152777988809406380D+00 &  0.10578621411215754890D-08 \\
      0.18566756593083182116D+00 &  0.28424278401127180014D-08 \\
     -0.35199891174698963256D+00 & -0.24190974576544113333D-05 \\
      0.74514028477189295514D+00 &  0.12627505833721768389D-03 \\
     -0.15698688665407300795D+01 & -0.36906906699773089097D-02 \\
      0.24402663496436078603D+01 &  0.71150880132384387316D-01 \\
      0.42564778384044014501D+01 & -0.92028028234742281569D+00 \\
     -0.86628781485086884613D+02 &  0.69659670041612025801D+01 \\
      0.76780431897157477513D+03 &  0.12336022855791130638D+01 \\
     -0.56054291413858448509D+04 & -0.86886243977848514246D+03 \\
      0.34709760011543563256D+05 &  0.13997432791721954713D+05 \\
     -0.17268737514629848298D+06 & -0.13922909431496294516D+06 \\
      0.61187840164952315907D+06 &  0.10063302790914770506D+07 \\
     -0.90229280253909143653D+06 & -0.55090070720520207541D+07 \\
     -0.58239693510024439315D+07 &  0.22810067701144771236D+08 \\
      0.55585305303478347757D+08 & -0.68226520995783255732D+08 \\
     -0.26301929032497527308D+09 &  0.12293641811239037282D+09 \\
      0.84052909616481142218D+09 &  0.21196827324494884533D+08 \\
     -0.18657792154606295188D+10 & -0.10131267803607086455D+10 \\
      0.25915388904113076082D+10 &  0.38427156992160196076D+10 \\
     -0.92517922470966062167D+09 & -0.86034302033451782637D+10 \\
     -0.51036926935170978903D+10 &  0.12626524314344868926D+11 \\
      0.13657170819172273067D+11 & -0.11301642152404259095D+11 \\
     -0.18234815409451008101D+11 &  0.35546723514474331442D+10 \\
      0.14355121657220596589D+11 &  0.45839423517954065893D+10 \\
     -0.61032052837541757328D+10 & -0.64509344322401839565D+10 \\
      0.84275175499628369228D+09 &  0.32773844995198787342D+10 \\
      0.15857079566801991882D+09 & -0.60570352545416858098D+09 \\
 \bottomrule
 \end{tabular}
 \label{tab2b}
 \end{table}

 \begin{table}[th]
     \caption{Similar to Table~\ref{tab2a}, $c_j$ ($j=0,\, \ldots,\,
       19$) for $J_{2}(z)$ in $Q_3=\{z\in \mathbb{C}\, |\,
       \operatorname{Re}(z)\ge 0, \operatorname{Im}(z)\ge 0, 15\le |z|
       \le 120\}$. $N_2=0$.
} \scriptsize \centering
 \begin{tabular}{SS}
 \toprule
 {real part} & {imaginary part}\\
 \midrule
      0.10000000000000000268D+01 &  0.16274386801955295004D-16 \\
      0.14166666666666607632D+01 & -0.10287883462694876761D-13 \\
     -0.12152777777773596334D+00 &  0.21278064340420781315D-11 \\
      0.18566743838222558548D+00 & -0.21346739588855205915D-09 \\
     -0.35199453421212974648D+00 &  0.10810020949854670351D-07 \\
      0.74505620178581604077D+00 & -0.12758335415838672536D-06 \\
     -0.15694400222283724037D+01 & -0.19194750913862712585D-04 \\
      0.24655058999322356159D+01 &  0.14693073094005325564D-02 \\
      0.33607089142515892731D+01 & -0.57088317524506061035D-01 \\
     -0.69853332457831300099D+02 &  0.14378014381758075690D+01 \\
      0.55610818499114251623D+03 & -0.24643410535733376707D+02 \\
     -0.37382459971986918636D+04 &  0.27993273451335282871D+03 \\
      0.23868572841373643577D+05 & -0.17730529258353887568D+04 \\
     -0.14415275031702837209D+06 & -0.95697379334253398976D+03 \\
      0.75853591811873756981D+06 &  0.14271567167383035412D+06 \\
     -0.30972861668769145538D+07 & -0.15003099051227929984D+07 \\
      0.85376658237794416873D+07 &  0.84572529844164518814D+07 \\
     -0.12288668565761833934D+08 & -0.27945310142205107664D+08 \\
      0.26869963259519375012D+06 &  0.49973783788775688447D+08 \\
      0.16376853579174702084D+08 & -0.35949830322064479962D+08 \\
 \bottomrule
 \end{tabular}
 \label{tab2c}
 \end{table}

\bibliographystyle{plain}
\bibliography{abram4d}

\end{document}